\documentclass[a4paper,reqno,12pt]{amsart}
\usepackage{amsmath}
\usepackage{amsthm}
\usepackage{amscd}
\usepackage{amssymb}
\usepackage{amsfonts}
\usepackage{latexsym}
\usepackage[mathscr]{eucal}

\newtheorem{thm}{Theorem}[section]
\newtheorem{prop}[thm]{Proposition}
\newtheorem{lemma}[thm]{Lemma}
\newtheorem{cor}[thm]{Corollary}
\newtheorem{defn}[thm]{Definition}

\newtheorem{rem}{Remark}[section]

\def\O{{\mathcal O}}

\def\K{{\mathcal K}}

\def\T{{\mathcal T}}

\def\F{{\mathcal F}}
\def\L{{\mathcal L}}

\def\bT{{\Bbb T}}

\def\cst{{C${}^*$}}

\newcommand{\hcomp}{\hat{\mathbb{C}}}

\newcommand{\cL}{{\mathcal L}}
\newcommand{\cO}{{\mathcal O}}

\begin{document}
\title{KMS states on finite-graph C*-algebras}
\author{Tsuyoshi Kajiwara}
\address[Tsuyoshi Kajiwara]{Department of Environmental and
Mathematical Sciences,
Okayama University, Tsushima, 700-8530,  Japan}

\author{Yasuo Watatani}
\address[Yasuo Watatani]{Department of Mathematical Sciences,
Kyushu University, Motooka, Fukuoka, 819-0395, Japan}
\maketitle
\begin{abstract}
We study KMS states on
finite-graph \cst -algebras with sinks and sources. We compare 
finite-graph \cst -algebras with \cst -algebras associated with
complex dynamical systems of rational functions.  We show that
if the inverse temperature $\beta$ is large, then the set of extreme 
$\beta$-KMS states is parametrized by the set of sinks of the graph. 
This means that the sinks of a graph correspond to the 
branched points of a rational funcition from the point of KMS states. 
Since we consider graphs with sinks and sources, left actions 
of the associated bimodules are not injective. Then the associated 
graph \cst -algebras are realized as (relative) Cuntz-Pimsner algebras
studied by Katsura.  
We need to generalize Laca-Neshevyev's theorem of the construction of 
KMS states on Cuntz-Pimsner algebras to the case that left actions 
of bimodules are not injective.


\medskip\par\noindent
AMS SUBJECT CLASSIFICATION: 46L08, 46L55

\end{abstract}

\section{Introduction}
KMS states on \cst -algebras are originated from the equilibrium states
in statistical physics.  
Olsen-Pederson \cite{OP} studied KMS states on Cuntz-algebra
$\cO_{n}$ (\cite{Cu}) of $n$ generators with respect to the gauge action, 
and they proved that
a $\beta$-KMS state exists if and only if $\beta = \log n$ and the 
$\beta$-KMS state is unique. Evans \cite{Ev} extended the result  
to certain quasi-free automorphisms. 
Enomoto-Fujii-Watatani \cite{EFW} studied KMS states on
Cuntz-Krieger algebras $\cO_{A}$ (\cite{CK}) 
associated with finite graphs with no sinks nor sources. 
If the $0$-$1$ matrix $A$ is irreducible and not a permutation, then
a $\beta$-KMS state exists if and only if $\beta = \log r(A)$ and the 
$\beta$-KMS state is unique, where $r(A)$ is the spectral radius of $A$. 
Exel-Laca \cite{EL} studied KMS states on Exel-Laca algebras and 
Toeplitz extensions. They introduced  KMS states of finite type and infinite 
type, which are useful in our study. They showed that there occur 
phase transitions. Exel \cite{Ex} considered KMS states on 
his  $C^*$-algebras by endomorphism with a transfer operator. 
Kumjian and Renault \cite{KR}  studied 
KMS states on $C^*$-algebras associated with expansive maps.

Pimsner \cite{Pi} introduced a general construction of \cst -algebras
through Hilbert \cst -bimodules or $C^*$-correspondances. 
Many \cst -algebras are known to be
expressed as Cuntz-Pimsner \cst -algebras.

The above results on KMS states are extended to 
KMS states on \cst -algebras associated with subshifts 
in Matsumoto-Watatani-Yoshida \cite{MYW} and  
Cuntz-Pimsner algebras associated with bimodules of finite basis 
in Pinzari-Watatani-Yonetani \cite{PWY}. Laca-Neshevyev \cite{LN} 
gave a 
theorem of construction of KMS
states for general Cuntz-Pimsner algebras. Using their theorem, we 
 classified KMS states on  \cst -algebras associated with the
complex dynamical systems on the Riemann sphere 
$\hcomp$ given by iteration 
of rational functions $R$
and \cst -algebras associated with self-similar sets in \cite{IKW} 
with Izumi. In particular we showed that there exists a phase transition at 
$\beta = \log \deg R$. If the inverse temperature 
$\beta > \log \deg R$,   then the set of extreme 
$\beta$-KMS states is parametrized by the set of branched points. 

On the other hand, Cuntz-Krieger algebras are 
generalized as graph \cst -algebras associated with general graphs 
having sinks and sources, which are  studied for example in 
Kumujian-Pask-Raeburn \cite{KPR}, 
Kumujian-Pask-Raeburn-Renault \cite{KPRR} 
and Fowler-Laca-Raeburn \cite{FLR}.   
They consider relations  between graphs and 
the associated graph \cst -algebras. 
See \cite{R} by I. Raeburn  to know a 
total aspect of graph \cst -algebras. 

In this paper we study KMS states on 
finite-graph \cst -algebras associated with graphs 
having sinks and sources. 
We show that if the inverse temperature $\beta$ is large, 
then the set of extreme 
$\beta$-KMS states is parametrized by the set of sinks of the graph.
We compare 
finite-graph \cst -algebras with \cst -algebras associated with
complex dynamical systems of rational functions.  
Our result suggests that the sinks of a graph correspond to the 
branched points of a rational funcition from the point of KMS states.

Relative Cuntz-Pimsner algebras are a generalization of Cuntz-Pimsner
algebras and are defined in \cite{MS} and  studied in \cite{FMR}.
As in Katsura \cite{Kat2}, the graph $C^*$-algebras associated with
graphs having sources and sinks can be constructed as relative
Cuntz-Pimsner algebras of bimodules such that left actions 
are not injective. Hence we shall generalize 
Laca-Neshevyev's theorem of KMS states to that 
of relative Cuntz-Pimsner algebras associated with general 
\cst -correspondences
with a countable basis. Our proof is more constructive than
that of Laca-Neshevyev.   We need to 
investigate the structure of cores of relative Cuntz-Pimsner algebras 
to study it.

The contents of the present paper is as follows.  In section 2, we present
the fundamental matters of \cst -correspondences, relative Cuntz-Pimsner algebras
and the structure of cores of relative Cuntz-Pimsner algebras.  In section
3, we present properties of countable basis, the degree of
\cst -correspondences.  We prove a theorem of construction of KMS states
on relative Cuntz-Pimsner algebras which generalize the theorem of
Laca-Neshevyev.  In section 4, we present a classification of KMS
states on finite-graph \cst -algebras, and show that sinks correspond 
to KMS states if the inverse temperature is sufficiently large.

\section{\cst -correspondences and the structure of Cores}
In this section, we present fundamental matters of \cst -correspondences,
the construction of associated \cst -algebras, 
and investigate the structure of the
cores of relative Cuntz-Pimsner algebras using some results of Katsura 
\cite{Kat2},
\cite{Kat3}.

\begin{defn}
Let $A$ be a \cst -algebra. A linear space $X$ 
is called a Hilbert $A$-module if the following conditions hold:
\begin{enumerate}
  \item There exist an $A$-valued hermitian, positive definite inner
        product $(\cdot | \cdot)_A$ and a right action of $A$ which is
        compatible with the $A$-inner product.
  \item $X$ is complete with respect to the norm $\|x\| =
        \|(x|x)_A\|^{1/2}$.
\end{enumerate}
If the linear span of $A$-inner product is dense in $A$,then $X$ is called
full.
\end{defn}

Let $A$ be a $C^*$-algebra and  $X$ a Hilbert $A$-module. 
 We denote by $\cL(X)$ the set of linear operators on $X$ which are
adjointable with respect to the $A$-valued inner product.
For $x$ and $y \in X$, put $\theta_{x,y}z = x(y|z)_A$ for $z \in X$.   We
denote by $\K(X)$ the norm closure of the linear span of $\{\,\theta_{x,y}\,
| x, y \in X \}$ in $\cL(X)$.  If there exists a *-homomorphism
$\phi$ from $A$ to $\cL(X)$, then we call the pair $(X,\phi)$
(or simply $X$) a \cst-correspondence over $A$.
We assume neither that $X$ is full, that $\phi$ is non-degenerate nor that
$\phi$ is isometric.  Let $J_X = \phi^{-1}(\K(X)) \cap (\ker 
\phi)^{\perp}$,
and $J$ be a closed two sided ideal of $A$ contained in $J_X$.

A representation $\pi$ of a \cst -correspondence $(X,\phi)$ on a Hilbert
space ${\mathcal H}$ consists of representations $\pi_A$ and $\pi_X$ of
$A$ and $X$ i.e. $\pi_A$ is a *-homomorphism from $A$ to $B({\mathcal 
H})$ and
$\pi_X$ is a linear map from $X$ to $B({\mathcal H})$ satisfying
\[
  \pi_X(x)^*\pi_X(y) = \pi_A((x|y)_A), \quad
  \pi_X(x)\pi_A(a)  = \pi_X(xa), \qquad
  \pi_A(a)\pi_X(x)  = \pi_X(\phi(a)x),
\]
for $x$, $y \in X$ and $a \in A$.
When $\pi_A$ is injective, $\pi_X$ is isometric.

For a representation $\pi=(\pi_A,\pi_X)$ of $(X,\phi)$, there corresponds
  a representation $\pi_K$ of $\K(X)$ satisfying
$\pi_K(\theta_{x,y}) = \pi_X(x)\pi_X(y)^*$  \cite{KPW1}.
The representation $\pi_X^{(n)}$ of $X^{\otimes n}$, $\pi_K^{(n)}$
of $\K(X^{\otimes n})$ are also defined naturally.  We use the
notation $\K(X^{\otimes 0})=A$ and $\pi_K^{0} = \pi_A$ for convenience.

\begin{defn}(Fowler-Muhly-Raeburn \cite{FMR}, 
Katsura \cite{Kat3}) \label{def:J-covariance}
Let $J$ be a closed two sided ideal of $A$ contained in $J_X$.
A representation $\pi=(\pi_A,\pi_X)$ of a \cst -correspondence $(X,\phi)$ 
is said to be $J$-covariant if
\begin{equation} \label{eq:J-covariance}
\pi_A(a) = \pi_K(\phi(a)) \quad \text{ for any } a \in J.
\end{equation}
\end{defn}

Let $\pi = (\pi_A,\pi_X)$ 
be the representation of $(X,\phi)$ which is universal
for all $J$-covariant representations.
The relative Cuntz-Pimsner algebra $\cO_X(J) = {\rm C}^*(\pi)$ is
 the $C^*$-algebra 
generated by $\pi_A(A)$ and $\pi_X(X)$ for the 
universal representation $\pi$. 
We note that $\pi_A$ of the universal representation $\pi$
is known to be injective (Katsura \cite{Kat2} Proposition 4.11).

\begin{lemma} \label{lem:J1} (T.Katsura \cite{Kat2} Proposition 3.3)
Let $\pi= (\pi_A,\pi_X)$ be a representation of $(X,\phi)$.
Assume that $\pi$ is $J$-covariant and $\pi_A$ is injective.
Take $a$ in $A$. 
If $\pi_A(a)$ is in  $\pi_K(\K(X))$, then $a$ is in $J_X$ 
and $\pi_A(a) = \pi_K(\phi(a))$.
\end{lemma}

\begin{lemma}\label{lem:J2} Let $\pi= (\pi_A,\pi_X)$ be 
a representation of $(X,\phi)$. 
Assume that $\pi$ is $J$-covariant and $\pi_A$ is injective.  Then
for $a \in A$, $\pi_A(a)$ is in  $\pi_K(\K(X))$ if and only if 
$a$ is in $J$.
\end{lemma}
\begin{proof}
For $a \in A$, assume that  $\pi_A(a) \in \pi_K(\K(X))$. 
By Lemma \ref{lem:J1}, we have $a \in J_X$ and $\pi_A(a) = 
\pi_K(\phi(a))$.

By Katsura \cite{Kat3} Corollary 11.4, if $(\pi_A,\pi_X)$ is a 
representation
of  $(X,\phi)$ satisfying the equation (\ref{eq:J-covariance}), we have
\[
  \{a \in A \,|\, \phi(a) \in \K(X),\, \pi_A(a)= \pi_K(\phi(a))\} = J.
\]
This shows the conclusion.
\end{proof}

We define subalgebras $B_n$ $n \ge 1$ and $B_0$ by
\[
  B_n = \pi_K^{(n)}(\K(X^{\otimes n})), \quad B_0 = \pi_A(A).
\]
These are \cst -subalgebras of $\cO_X(J)$.
We put
\[
  \F^{(n)} = B_0 + B_1 + \cdots + B_n.
\]
For integers $n$, $i$ we introduce the notation $(n,i)$ by
\[
  (n,i) =
  \begin{cases}
    n-i \qquad n \ge 1,\,i \ge 1  \\
    n-1 \qquad n \ge 1,\,i=0 \\
    0  \qquad n=0,\, i=0.
  \end{cases}
\]
Let $k \in \K(X^{\otimes i})$ $i\ge 1$.
For $\xi_1 \in X^{\otimes i}$, $\xi_2 \in X^{\otimes n-i}$, we define
$k \otimes id_{(n,i)}$ by
\[
  (k \otimes id_{(n,i)})(\xi_1 \otimes \xi_2)
  = k \xi_1 \otimes \xi_2.
\]
Then $k \otimes id_{(n,i)}$ is an element of $\cL(X^{\otimes n})$.
The notation $a \otimes id_{(n,0)}$ means
$\phi(a)\otimes id_{n-1}$.  When $n=0$, $a \otimes id_{(0,0)}$ is
a left multiplication representation of $A$ on a Hilbert $A$-module $A$.

\begin{lemma} \label{lem:core}
For each $m$,  $B_m$ is an ideal in $\F^{(m)}$, and $\F^{(m)}$ is a
\cst -subalgebra.
\end{lemma}
\begin{proof}
We assume $1 \le m \le n$, $k \in \K(X^{\otimes
  m})$, $k' \in \K(X^{\otimes n})$.  Since $k \otimes id_{(n,m)} \in
  \cL(X^{\otimes n})$, we have
$(k \otimes id_{(n,m)})k' \in \K(X^{\otimes n})$.  By Katsura
  \cite{Kat2} Lemma 5.4, we have
\[
  \pi_K^{(m)}(k) \pi_K^{(n)}(k') = \pi_K^{(n)}((k \otimes id_{(n,m)})k').
\]
This shows that $B_n$ is an ideal in $\F^{(n)}$.  We can check the case 
$m=0$
separately.
\end{proof}

We need to investigate $\F^{(n-1)}\cap B_n$ for a proof of the
theorem of constructing KMS states.

Note that $(X^{\otimes n}J)_A$ is a right $A$-submodule of $X_A$.
By considering the embedding of "rank-one"  operators, 
we may have an inclusion 
$\K(X^{\otimes n}J) \subset \K(X^{\otimes n})$, and we have
$\theta_{\xi \cdot j, \eta \cdot j'} \in K(X^{\otimes n}J)$ for
$j$, $j' \in J$.

\begin{lemma}
An element $k \in \K(X^{\otimes n})$ is in  $\K(X^{\otimes n}J)$
if and only if
\[
  (\xi | k \eta)_A \in J \qquad \forall \xi, \eta \in X^{\otimes n}.
\]
\end{lemma}
\begin{proof}
We refer Fowler-Muhly-Raeburn \cite{FMR} Lemma 1.6 and Katsura \cite{Kat3}
for quotient modules $X_J$.  
The notation $[a]_J \in A/J$ for an element $a$ of
a \cst -algebra $A$ means the quotient image of $a \in A$ by $J$.

Since $T \in \cL(X)$ leaves $XJ$ invariant, we can consider 
an operator $[T]_J \in  
 X/XJ=X_J$. 
The map $k \in K(X^{\otimes n}) \to [k]_J$  is an onto map from 
$K(X^{\otimes n})$
to $K(X^{\otimes n}{}_J)$, and its kernel is $k$'s such that
$k \in \K(X^{\otimes n}J)$ (\cite{Kat3} Lemma 1.6).  Then $k \in 
\K(X^{\otimes n})$
is contained in $\K(X^{\otimes n}J)$ if and only if $[k]_J=0$.
Moreover we have
\begin{align*}
  ( \xi | k\eta)_A \in J   \quad &\text{if and only if} \quad
    [(\xi |k \eta)_A]_{J} = 0 \\
   \quad &\text{if and only if} \quad  ([\xi]_J | [k]_J[\eta]_J)_{A/J} = 
0.
\end{align*}
If it holds for each $\xi$, $\eta$, then we have $[k]_J =0$, and
this means $k \in \K(X^{\otimes n}J)$.
\end{proof}

We put $B_n' = \pi_K^{(n)}(\K(X^{\otimes n}J))$, $(n \ge 1)$ and
$B_0' = \pi_A(J)$.
For the case $J = J_X$, 
the following Lemma is presented in Katsura \cite{Kat2}.  It also holds
for the case $J \subset J_X$. 

The following lemmas are T.Katsura \cite{Kat2} Lemma 5.10 and
T. Katsura \cite{Kat2} Proposition 5.11 for the case $J=J_X$.
The proof for general cases is the same as the case $J=J_X$.

\begin{lemma}\label{lem:approximate_unit}
Let $k \in \K(X^{\otimes n+1})$. Then  for an approximate unit
$\{u_{\lambda }\}_{\lambda \in  \Lambda}$ in $\K(X^{\otimes n})$,
we have that $k = \lim_{\lambda \in \Lambda}(u_{\lambda} \otimes id_1)k$.
\end{lemma}

\begin{lemma}\label{lem:subset}
We have that $\F^{(n)} \cap B_{n+1} \subset B_n$.
\end{lemma}

\begin{prop} We have  that $B_n \cap B_{n+1} = B_n'$ and 
$\F^{(n)} \cap B_{n+1} = B_n'$.
\end{prop}
\begin{proof} First we show that $B_n \cap B_{n+1}=B_{n}'$.
This is  Katsura \cite{Kat2} Proposition 5.9 for $J=J_X$.
Let $n=0$.  Then we have
\[
  \pi_A(A) \cap B_1 = \pi_A(A) \cap \pi_K(\K(X)),
\]
By Lemma \ref{lem:J2}, we have  $\pi_A(A) \cap B_1 = \pi_A(J)$.
Moreover the proposition holds for $n=0$ because $B_0' = \pi_A(J)$.
We may assume $n \ge 1$.  Let $a$, $b \in J$, $\xi$, $\eta \in X^{\otimes 
n}$.
Then we have $\pi_A(a) \in B_1 = \pi_K(\K(X))$, and
\begin{align*}
  \pi_K^{(n)}(\theta_{\xi a,\eta b}) & = \pi_X^{(n)}(\xi a)
  \pi_X^{(n)}(\eta b)^* \\
   & = \pi_X^{(n)}(\xi) \pi_A(a)\pi_A(b)^* \pi_X^{(n)}(\eta)^*.
\end{align*}
Since $\pi_A(a)\pi_A(b)^* \in B_1$, the left hand side is contained in
$B_{n+1}$.  Then we have
\[
  B_n' \subset B_n \cap B_{n+1}.
\]

Let $x \in B_n \cap B_{n+1}$.  There exists $k \in \K(X^{\otimes n})$
such that $\pi_K^{(n)}(k) =x$.  For $\xi$, $\eta \in X^{\otimes n}$,
we have
\begin{align*}
  \pi_A((\xi|k\eta)_A)
  & = \pi_X^{(n)}(\xi)^* \pi_K^{(n)}(k) \pi_X^{(n)}(\eta) \\
  & = \pi_X^{(n)}(\xi)^* x \pi_X^{(n)}(\eta).
\end{align*}
By $x \in B_{n+1}$, the last expression is contained in $B_1$.
Since for  $\xi$,  $\eta \in X^{\otimes n}$,
$\pi_A((\xi|k\eta)_A)\in B_1$, we have $(\xi|k \eta)_A \in J$
for  $\xi$,  $\eta \in X^{\otimes n}$ by Lemma \ref{lem:J2}.
  Then we have
$k \in K(X^{\otimes n}J)$, and we have $x=\pi_K^{(n)}(k) \in B_n'$.

Lastly,  we shall show that $\F^{(n)} \cap B_{n+1} = B_n'$.
By Lemma \ref{lem:subset}, $\F^{(n)}\cap B_{n+1} \subset B_n$. We 
have
\begin{align*}
  \F^{(n)} \cap B_{n+1} & = (\F^{(n)} \cap B_{n+1}) \cap B_n 
                         = (\F^{(n)} \cap B_n) \cap B_{n+1} \\
                        & = B_n \cap B_{n+1} 
                         = B_n'.
\end{align*}
This completes the  proof. 
\end{proof}

\section{KMS states on relative Cuntz-Pimsner algebras}
In this section, we generalize a theorem of the construction of KMS states
of Cuntz-Pimsner algebras in Laca-Neshevyev \cite{LN} to relative 
Cuntz-Pimsner algebras.

Let $A$ be a $\sigma$-unital \cst -algebra and $X$ be a countably
generated Hilbert $A$-module.

\begin{defn}
A sequence $\{u_i\}_{i=1}^{\infty}$ of a Hilbert (right) $C^*$-module 
$X$ over $A$ is called
a countable basis (or normalized tight frame) of $X$  if
\begin{equation} \label{eq:basis1}
  x = \sum_{i=1}^{\infty} u_i(u_i|x)_A
\end{equation}
for each $x \in X$, where the right hand side converges in norm. 
\end{defn}

As in Remark after \cite{KPW2} Proposition 1.2 ,
\begin{equation}\label{remark}
\text{For }  a, b \in \K(X), x \in X 
\text{ with } 0 \le a \le b \le I, \quad
\|x-bx\|^2 \le \|x\|\|x-ax\|. 
\end{equation}
This inequality implies that the right hand side of (\ref{eq:basis1}) 
converges unconditionally in the following sense: 
For every $\varepsilon > 0$, there exists a finite subset $F_0$
of ${\mathbb N}$ such that for every finite subset $F$ of ${\mathbb N}$ with
$F_0 \subset F$ we have
\[
   \|x - \sum_{i\in F}u_i(u_i|x)_A \| < \varepsilon.
\]
Since $\sum_{i \in F}\theta_{u_i,u_i}  \le I$ for each finite subset
$F \in {\mathbb N}$, it is sufficient to prove (\ref{eq:basis1}) for each $x$
in some norm dense subset of $X$.
We often write it as $x = \sum_{i\in {\mathbb N}}u_i(u_i|x)_A $ to 
express unconditinally convergence. More generally, 
for any countable set $\Omega$,  the notation 
\begin{equation} \label{eq:basis}
   x = \sum_{i\in \Omega}u_i(u_i|x)_A.
\end{equation}
makes sense as  unconditional convergence. 

We can show the following Lemma: 

\begin{lemma} Let $A$ be a \cst -algebra, $Y$ a \cst -correspondence
over $A$ and $X$ a Hilbert $A$-module.  Let $\{u_i\}_{i \in \Omega_1}$
be a countable basis of $X$ and $\{v_j\}_{j \in \Omega_2}$ a countable
basis of $Y$.  Then $\{u_i \otimes v_j\}_{(i,j)
\in \Omega_1 \times \Omega_2}$ is a countable basis of the inner tensor
product module $X \otimes_A Y$ of $X$ and $Y$.
\end{lemma}
\begin{proof} Let $\varepsilon >0$.  We fix a nonzero $x \otimes y \in X
\otimes_A Y$.  Let 
$\delta = \varepsilon^2/\|x \otimes y\|$.
be a positive number. 
We take a finite subset $F$ of $\Omega_1$ such that
\[
  \| \sum_{i \in F} u_i(u_i|x)_A -x \| < \frac{\delta}{2\|y\|}.
\]
Put $s$ be the cardinality of $F$.
For each $i$ $(i=1,\dots,s)$ we take a finite subset $G_i \subset \Omega_2$
such that if $G'$ is a finite subset containing $G_i$ then it holds
that
\[
  \|\sum_{j \in G'} v_j(v_j|(u_i|x)_Ay)_A - (u_i|x)_Ay \|
     < \frac{\delta}{2s\|u_i\|}.
\]
Let $G$ be a finite subset containing $\bigcup_{i=1}^{s}G_i$.
Then we have
\begin{align*}
  & \|x \otimes y - \sum_{(i,j)\in F \times G}
   u_i\otimes v_j(u_i \otimes v_j|x \otimes y)_A \| \\
  = & \|x\otimes y - \sum_{i \in F}\sum_{j \in G} u_i
    \otimes v_j(v_j|(u_i|x)_Ay)  \|\\
  \le & \| x \otimes y - \sum_{i \in F} u_i(u_i|x)_A\otimes y \|
   + \| \sum_{i \in F} u_i\otimes (u_i|x)_Ay
      - \sum_{i \in F} \sum_{j \in G} u_i \otimes v_j(v_j|(u_i|x)_Ay)_A \|\\
   \le & \|x-\sum_{i \in F_0}u_i(u_i|x)_A\|\|y\|
      + \sum_{i \in F}\|u_i\|
   \|\sum_{j \in G}  v_j(v_j|(u_i|x)_Ay)_A - (u_i|x)_Ay  \| \\
   <  & \delta.
\end{align*}
Using (\ref{remark}), for each finite
subset $H$ of $\Omega_1 \times \Omega_2$ such
that $H \supset F\times G$ we have  that
\[
\|x \otimes y - \sum_{(i,j)\in H}
   u_i\otimes v_j(u_i \otimes v_j|x \otimes y)_A \| < \varepsilon.
\]
Hence $x \otimes y = 
\sum_{(i,j)\in \Omega_1 \times \Omega_2}
   u_i\otimes v_j(u_i \otimes v_j|x \otimes y)_A$.  
If $z=\sum_{p \text{ finite}} x_p \otimes y_p$, then 
\[
  \sum_{(i,j) \in \Omega_1 \times \Omega_2} u_i \otimes v_j
     (u_j \otimes v_j | z)_A = z.
\]
Since the subset of elements of the form
$\sum_{p \text{ finite}} x_p \otimes y_p\,$ is
dense in $X \otimes_A Y$, $\{u_i \otimes v_j)\}_{\Omega_1 \times \Omega_2}$
constitute a basis of $X\otimes_A Y$.
\end{proof}

We fix a \cst -correspondence $X$ over a \cst -algebra $A$, and a countable
basis $\{u_i\}_{i=1}^{\infty}$ of $X$.
Let $J$ be a closed two sided ideal of $A$ which is contained in $J_X$. 
  $\O_X(J)$ denotes
the relative Cuntz-Pimsner algebra constructed from $X$ and $J$ in 
section 2.

\begin{lemma}\label{lem:deg} Let $\tau$ be a tracial state on $A$.
Then the possibly infinite positive number $\sup_n 
\sum_{i=1}^{n}\tau((u_i|u_i)_A)$
does not depend on the choice of a countable basis 
$\{u_i\}_{i=1}^{\infty}$.
\end{lemma}
\begin{proof} Let $\{v_i\}_{j=1}^{\infty}$ be another countable basis of 
$X$.
Since $u_i = \lim_{m \to \infty} \sum_{j=1}^{m} v_j(v_j|u_i)_A$, 
we have
\begin{align*}
 & \sup_n \sum_{i=1}^{n}\tau((u_i|u_i)_A)
     = \sup_n \lim_{m \to \infty} \sum_{i=1}^{n}\sum_{j=1}^m
         \tau ((v_j(v_j|u_i)_A|u_i)_A) \\
= & \sup_n \sup_m \sum_{i=1}^{n}\sum_{j=1}^m
       \tau((v_j|u_i)_A^* (v_j|u_i)_A)
 = \sup_m \sup_n \sum_{i=1}^{n}\sum_{j=1}^m
          \tau((v_j|u_i)_A (v_j|u_i)_A^* ) \\
= & \sup_m \lim_{n \to \infty} \sum_{i=1}^{n}\sum_{j=1}^m
       \tau ( (v_j|u_i)_A (u_i|v_j)_A )
 = \sup_m \lim_{n \to \infty}\sum_{j=1}^m \sum_{i=1}^{n}
       \tau((v_j|u_i(u_i|v_j)_A)_A) \\
= & \sup_m \sum_{j=1}^m   \tau((v_j|v_j)_A).
\end{align*}
\end{proof}
Therefore  we may put $d_{\tau} =  \sup_n \sum_{i=1}^{n}\tau((u_i|u_i)_A)
\le \infty$.

We denote by $\T(A)$ the set of tracial states on $A$.

\begin{defn} The degree $d(X)$ of a $C^*$-correspondence $X$ 
is defined by $d(X) := \sup \{d_{\tau} | \tau \in \T(A)\}$. 
We say that $X$ is of finite-degree type 
if $d(X)< \infty$. 
\end{defn}
\begin{lemma}
If $A$ is commutative, then $d(X)=\sup_n\|\sum_{i=1}^n(u_i|u_i)_A\|$ for
any countable basis $\{u_i\}_{i=1}^{\infty}$.
\end{lemma}
\begin{proof}
We assume that $A$ is commutative.
By the method similar as in the proof of Lemma \ref{lem:deg}, we can
show that $\sup_n \|\sum_{i=1}^{n} (u_i|u_i)_A\|$ does not depend on
the choice of a countable basis $\{u_i\}_{i=1}^{\infty}$.

Since
\[
  \sum_{i=1}^{n}\tau((u_i|u_i)_A)
  = \tau(\sum_{i=1}^{n}(u_i|u_i)_A)
  \le \|\sum_{i=1}^{n}(u_i|u_i)_A)\|,
\]
we have $d(X) \le \sup_n\|\sum_{i=1}^{n}(u_i|u_i)_A)\|$.

We fix a countable basis $\{u_i\}_{i=1}^{\infty}$.
For every $\varepsilon >0$, there exists an $n_0$ such
that for each $n \ge n_0$,
\[
  \|\sum_{i=1}^n (u_i|u_i)_A\| > \sup_n  \|\sum_{i=1}^n (u_i|u_i)_A\|
    -\varepsilon.
\]
There exists a tracial state $\tau$ such that
\[
\tau(\sum_{i=1}^{n}(u_i|u_i)_A) >  \|\sum_{i=1}^{n}(u_i|u_i)_A\|
                                     -\varepsilon.
\]
Thus we have $d(X) \ge \sup_n\|\sum_{i=1}^{n}(u_i|u_i)_A\|$.
\end{proof}

Let $R$ be a rational function of degree $N$ and $A$ a commutative \cst 
-algebra
${\rm C}({\hcomp})$.  Consider a \cst -correspondence $X$ over $A$ associated
with the complex dynamical system given by $R$ on $\hcomp$.
As described in \cite{K}, we can choose a concrete 
countable basis such that  we can compute explicitly as
\[
\sum_{i=1}^{n} (u_i|u_i)_A(y)= \verb!#!\{R^{-1}(y)\}.
\]
This equation is also shown in \cite{KW2} for any basis. 
Thus  we have
\[
\sup_n \|\sum_{i=1}^{\infty} (u_i|u_i)_A\| = N.
\]
Therefore the degree of $X$ coincides with the degree of $R$.  
Similar formulas hold for the case of self-similar maps. 

Let $Y$ be a \cst -correspondence over a \cst -algebra $B$  
of finite-degree type with $d(Y)=N$. 
Let $\beta > \log N$.  For a tracial state $\tau$ on $B$,
we can define a bounded linear functional $\hat{\tau_1}$ on $\cL(Y)$ by
\[
  \hat{\tau}_1(k) = e^{-\beta} \sum_{i=1}^{\infty}\tau((u_i|Tu_i)_A),
\]
for $T \in \cL(Y)$.

We need an elementary fact as follows:  
\begin{lemma}
$\hat{\tau_1}$ is a trace and does not depend on the choice of
a basis $\{u_i\}_{i=1}^{\infty}$.
\end{lemma}
\begin{proof}
Let $\{v_j\}_{j=1}^{\infty}$ be another basis of $X$, and
$T$ be a positive element in $\cL(Y)$.
  As in the proof of Lemma \ref{lem:deg}, we have
\begin{align*}
  \sup_n \sum_{i=1}^{n}e^{-\beta}\tau((T^{1/2}u_i|T^{1/2}u_i)_A)
    & = \sup_n \lim_{m \to \infty} \sum_{i=1}^{n}\sum_{j=1}^m
         e^{-\beta}\tau ((v_j(v_j|T^{1/2}u_i)_A|T^{1/2}u_i)_A) \\
& = \sup_m \lim_{n \to \infty}\sum_{j=1}^m \sum_{i=1}^{n}
       e^{-\beta}\tau((v_j|T^{1/2}u_i(T^{1/2} u_i|v_j)_A)_A) \\
& = \sup_m \lim_{n \to \infty}\sum_{j=1}^m \sum_{i=1}^{n}
       e^{-\beta}\tau((T^{1/2}v_j|u_i( u_i|T^{1/2}v_j)_A)_A) \\
& = \sup_m \sum_{j=1}^m e^{-\beta}  \tau((T^{1/2}v_j|T^{1/2}v_j)_A).
\end{align*}
This shows that the definition of $\hat{\tau_1}$ does not depend on the
choice of basis.
Let $U$ be a unitary in $\cL(Y)$.  Then $\{Uu_i\}_{i=1}^{\infty}$ is
also a basis of $Y$.  For $T \in \cL(Y)$, we have 
\[
  \sum_{i=1}^{\infty}e^{-\beta}\tau(Uu_i|TUu_i)_A
  = \sum_{i=1}^{\infty}e^{-\beta}\tau(u_i|Tu_i)_A.
\]
Then  $\hat{\tau_1}(U^*TU) = \hat{\tau_1}(T)$,
and it follows that $\hat{\tau_1}$ is a
trace.
\end{proof}

Let $I$ be an closed two sided ideal of a \cst -algebra $B$, and 
$\varphi$ be
a state on $I$. Consider the GNS representation 
$(\pi_{\varphi},H_{\varphi}, \xi_{\varphi})$. Let 
$\pi : A \rightarrow B(H_{\varphi})$ be the extension of 
$\pi_{\varphi}$ to $A$.   
Recall that the canonical extension  
$\overline{\varphi}$ of $\varphi$ 
to $B$
is defined as $\overline{\varphi}(a) 
= (\pi(a)\xi_{\varphi},\xi_{\varphi})$. 
Then 
$\overline{\varphi}(a) = \lim_{i}\varphi(ae_i)$, for any
approximate unit $\{e_i\}_i$ in $I$ as in \cite{Bl} Prop. 6.4.16. 

Let $\tau_1$ be the restriction of $\hat{\tau}_1$ on $\K(Y)$.  We note that
the \cst -algebra $\K(Y)$ is a closed two sided ideal of $\cL(Y)$.

\begin{lemma} \label{lem:natural}
The canonical extension  $\overline{\tau}_1$ of $\tau_1$
to $\cL(Y)$ is given by $\hat{\tau}_1$.
\end{lemma}

\begin{proof}
We note that $\hat{\tau}_1(\theta_{x,y})= e^{-\beta}\tau((y|x)_A)$. 
If $\{u_i\}_{i=1}^{\infty}$ is a basis of $Y$, then 
$\{\,\theta_{u_i,u_i}\}_{i=1}^{\infty}$ is an approximate unit in
$\K(Y)$. Therefore  
for $T \in \cL(Y)$, we have
\begin{align*}
  \overline{\tau}_1(T) & = \lim_{m \to \infty}\sum_{j}^{m}
  \tau_1(T\theta_{u_j,u_j}) 
                        = \lim_{m \to \infty}\sum_{j}^m \tau_1
       (\theta_{Tu_j,u_j}) \\
                      & = \lim_{m \to \infty} \sum_{j=1}^m e^{-\beta}
       \tau((u_j|Tu_j)_A) 
                       = \hat{\tau}_1(T)
\end{align*}

\end{proof}

Let $A$ be a \cst -algebra and $X$ be a \cst -correspondence over $A$
with a countable basis $\{u_i\}_{i=1}^{\infty}$.
Since we use tensor products of correspondences and their bases 
frequently, we
use the notations of multi index.  Namely, for ${\bf p}=
(i_1,i_2,\dots,i_n) \in {\mathbb N}^n$, we write ${\bf u}_{\bf p} = u_{i_1}
\otimes u_{i_2} \otimes \cdots \otimes u_{i_n}$.

We assume that $X$ is of finite-degree type.
We can also define a bounded tracial linear functional $\hat{\tau}^{(n)}$
on $\cL(X^{\otimes n})$ and its restriction $\tau^{(n)}$ to 
$\K(X^{\otimes n})$
using the Hilbert $A$-module $X^{\otimes n}$ and its basis
$\{{\bf u}_{\bf p} \}_{{\bf p} \in {\mathbb N}^n}$ as
\[
  \hat{\tau}^{(n)}(T) =e^{-n \beta} \sum_{{\bf p} \in {\mathbb N}^n}
       \tau(({\bf u}_{\bf p} | T{\bf u}_{\bf p} )_A) \qquad \text{for } T
       \in \cL(X^{\otimes n}).
\]


\begin{defn} Let $J$ be a closed two-sided ideal of $A$ such that
$J \subset J_X$, and $\beta$ be a positive real number.  A tracial state
$\tau$ on $A$ satisfies $\beta$-condition if it satisfies the following
two conditions:
\begin{enumerate}
  \item[($\beta$1)] $\sum_{i=1}^{\infty}\tau((u_i|\phi(a)u_i)_A) =
                   e^{\beta}\tau(a)$ \quad $\forall a \in J$,
  \item[($\beta$2)] $\sum_{i=1}^{\infty}\tau((u_i|\phi(a)u_i)_A) \le
                   e^{\beta}\tau(a)$ \quad  $\forall a \in A^+$.
\end{enumerate}
\end{defn}

Since $B_{n}$ is isomorphic to $\K(X^{\otimes n})$ by $\pi_K^{(n)}$
for each $n$, we can define a bounded linear tracial functional
$\sigma^{(n)}$ on $B_n$ by
\[
  \sigma^{(n)} = \tau^{(n)} \circ (\pi_K^{(n)})^{-1}.
\]
For convenience, we put $\tau^{(0)} = \tau$, $\sigma^{(0)}=
\tau \circ \pi_A{}^{-1}$.

\begin{prop}\label{prop:well}
We assume that a tracial state $\tau$ on $A$ satisfies $(\beta 1)$.
Then, for $x \in \F^{(n)} \cap B_{n+1} = B_n \cap B_{n+1}$, we have
\[
  \sigma^{(n+1)}(x) = \sigma^{(n)}(x).
\]
\end{prop}

\begin{proof}
We put ${\bf p}=(i_1,i_2,\dots,i_n)$, ${\bf u}_{\bf p} = u_{i_1}\otimes
u_{i_2} \otimes
\cdots \otimes u_{i_n}$, ${\bf p}'=(i_1,i_2,\dots,i_n,i_{n+1})$ and
${\bf u}_{{\bf p}'}= u_{i_1}\otimes u_{i_2} \otimes\cdots \otimes u_{i_n}
\otimes u_{i_{n+1}} = {\bf u}_{\bf p} \otimes u_{i_{n+1}}$.

Due to $x \in \F^{(n)} \cap B_{n+1} = B_n \cap B_{n+1}=B_n'$,
we can write as $x=\pi_K^{(n)}(k)$, $k \in K(X^{\otimes n}J)$, which
shows $({\bf u}_{\bf p} | k{\bf u}_{\bf p})_A \in J$ for each ${\bf 
u}_{\bf p}$.

On the other hand, we can write $x=\pi_K^{(n+1)}(k')$, $k' \in
  K(X^{\otimes n+1})$ because $x \in B_{n+1}$.
Then we have
\begin{align*}
   \pi_A(({\bf u}_{{\bf p}'} | k'{\bf u}_{{\bf p}'})_A)
  = &  \pi_X(u_{i_{n+1}})^* \cdots \pi_{X}(u_{i_1})^*x
  \pi_X(u_{i_1})\cdots \pi_X(u_{i_{n+1}}) \\
=  &  \pi_X(u_{i_{n+1}})^* \left( \pi_X(u_{i_n})^* \cdots
  \pi_{X}(u_{i_n})^* x \pi_X(u_{i_1}) \cdots \pi_X(u_{i_n})\right)
  \pi_X(u_{i_{n+1}}) \\
   =& \pi_X(u_{i_{n+1}})^*
  \left( \pi_A({\bf u}_{\bf p}|k {\bf u}_{\bf p})_A)\right)
  \pi_X(u_{i_{n+1}}) \\
   = &\pi_X(u_{i_{n+1}})^*
        \pi_X(\phi( ({\bf u}_{\bf p}|k {\bf u}_{\bf p})_A
)u_{i_{n+1}}) \\
   = &\pi_A((u_{i_{n+1}} | \phi(({\bf u}_{\bf p}|k {\bf u}_{\bf p})_A
   ) u_{i_{n+1}})_A).
\end{align*}
Then we have 
\[
  \tau^{(n+1)}(k')
= e^{-(n+1)\beta} \sum_{{\bf p} \in {\mathbb N}^n}
  \sum_{i_{n+1}=1}^{\infty}
\tau((u_{i_{n+1}}|\phi(({\bf u}_{\bf p}|k {\bf u}_{\bf p})_A
)u_{i_{n+1}})_A).
\]
Using ($\beta 1$)
\begin{align*}
  \tau^{(n+1)}(k')
  &  = e^{-n \beta} \sum_{{\bf p} \in {\mathbb N}^n}\tau(
({\bf u}_{\bf p}|k {\bf u}_{\bf p})_A ) \\
  & = \tau^{(n)}(k).
\end{align*}
By this, we have $\sigma^{(n+1)}(x) = \sigma^{(n)}(x)$ for
$x \in \F^{(n)} \cap B_{n+1}$.
\end{proof}

We assume that $\tau$ satisfies $(\beta 2)$.  $\tau^{(n+1)}$ is a
tracial bounded linear functional on $\K(X^{\otimes n+1})$.

We assume $n \ge 1$.  We denote by $\F(\Sigma)$ the set of finite
subsets of $\Sigma$.  Let $e_F = \sum_{{\bf p} \in F}
\theta_{{\bf u}_{\bf p}, {\bf u}_{\bf p}}$ 
for a finite subset $F$ of ${\mathbb N}^n$. 
Then $\{e_F\}_{F \in \F({\mathbb N}^n)}$ is an
approximate unit of $\K(X^{\otimes n})$.  The canonical extension 
$\overline{\tau^{(n)}}$ of $\tau^{(n)}$ to 
$\L(X^{\otimes n})$ satisfies 
\[
  \overline{\tau^{(n)}}(T)
  = e^{-n\beta} \lim_F \sum_{{\bf p} \in F}  \tau(({\bf u}_{\bf p}| T
e_{F}{\bf u}_{\bf p})_A),
\]
where $T \in \cL(X^{\otimes n})$, and
it is expressed by Lemma \ref{lem:natural} as
\[
  \overline{\tau^{(n)}}(T)   = e^{-n \beta}
  \sum_{{\bf q} \in {\mathbb N}^n} \tau(({\bf u}_{\bf q}|T
{\bf u}_{\bf q})_A)
\]
for $T \in \cL(X^{\otimes n})$.  Then the following Lemma holds.

\begin{lemma} \label{lem:natural0}
We assume that $\tau$ satisfis $(\beta 2)$. 
Let $n \ge 1$ and $0 \le i \le n$.  For $k \in \K(X^{\otimes i})$.
we have
\[
  \overline{\tau^{(n)}}(k \otimes id_{(n,i)})   = e^{-n \beta}
  \sum_{{\bf q} \in {\mathbb N}^n} \tau(({\bf u}_{\bf q}|(k \otimes 
id_{(n,i)})
{\bf u}_{\bf q})_A).
\]
\end{lemma}

Since $B_{n+1}$ is an ideal of $\F^{(n+1)}$, there exists the canonical 
extension $\overline{\sigma^{(n+1)}}$ on $\F^{(n+1)}$. For a finite 
subset $F$ of ${\mathbb N}^{n+1}$, 
we put $\hat{e}_F = \sum_{{\bf q} \in F}\pi_X^{(n+1)}({\bf u}_{\bf q})
  \pi_X^{(n+1)}({\bf u}_{\bf q})^* 
= \sum_{{\bf q} \in F}\pi_K^{(n+1)}
(\theta_{{\bf u}_{\bf q},{\bf u}_{\bf q}})$.  Then
$\{\hat{e}_F\}_{F \in \F({\mathbb N}^{n+1})}$ is an approximate unit
of $B_{n+1}$.  Then we have
\[
  \overline{\sigma^{(n+1)}}(x) = \lim_F \sigma^{(n+1)}(x\hat{e}_F),
\]
for $x \in \F^{(n+1)}$.

Let $x \in B_i$ for $0 \le i \le n$.  We write as $x = \pi_K^{(i)}(k)$ 
where
$k\in \K(X^{\otimes i})$.  Then we have
\[
  x\hat{e}_F =  \sum_{{\bf q} \in F}
   \pi_K^{(n+1)}((k \otimes id_{(n+1,i)})\theta_{{\bf u}_{\bf q},{\bf
   u}_{\bf q}}).
\]
Using this,
\begin{align*}
  \overline{\sigma^{(n+1)}}(x) =& \lim_F \sigma^{(n+1)}(x\hat{e}_F)  \\
  = &  \lim_F \tau^{(n+1)} \left((k \otimes id_{(n+1,i)})
  \sum_{{\bf q} \in F}\theta_{{\bf u}_{\bf q},{\bf  u}_{\bf
  q}}\right) 
  =  \overline{\tau^{(n+1)}}(k \otimes id_{(n,i)}).
\end{align*}

\begin{lemma} \label{lem:natural2}
We assume that $\tau$ satisfis $(\beta 2)$.
Let $x \in \F^{(n)}$ with  $x = \sum_{i=0}^{n}x_i$, where $x_i 
\in B_i$.
Take $k_i \in \K^{(i)}(X^{\otimes i})$ such that
$x_i = \pi_K^{(i)}(k_i)$.  Then we have
\[
  \overline{\sigma^{(n+1)}}(x)
  =  e^{-(n+1)\beta }\sum_{{\bf q} \in {\mathbb N}^{n+1} }\tau(({\bf u}_{\bf
  q} |  \sum_{i=1}^n (k_i \otimes id_{(n+1,i)}){\bf u}_{\bf q} )_A).
\]
\end{lemma}

\begin{proof}
Using Lemma \ref{lem:natural0},   we have
\begin{align*}
  \overline{\sigma^{(n+1)}}(x) = & \sum_{i=0}^n
  \overline{\sigma^{(n+1)}}(x_i) 
  =  \sum_{i=0}^n \overline{\tau^{(n+1)}}(k_i \otimes id_{(n+1,i)}) \\
  = & e^{-(n+1)\beta }\sum_{{\bf q} \in {\mathbb N}^{n+1}}\tau(({\bf u}_{\bf
  q} |
  \sum_{i=1}^n (k_i \otimes id_{(n+1,i)}){\bf u}_{\bf q} )_A).
\end{align*}
\end{proof}

\begin{prop}\label{prop:inequality}
We assume that $\tau$ satisfies $(\beta 2)$.
For $x \in (F^{(n)})^+$, we have
\[
   \overline{\sigma^{(n+1)}}(x) \le \overline{\sigma^{(n)}}(x).
\]
\end{prop}

\begin{proof}
We take $x \in (F^{(n)})^+$.  Then we can write as $x = y^*y$
where $y \in \F^{(n)}$.  We also write as
$y = \sum_{i=0}^ny_i$ where $y_i \in B_i$, and write as
$y_i = \pi_K^{(i)}(h_i)$, $h_i \in \K(X^{\otimes i})$.

Then, by Lemma 5.4 in \cite{Kat2},  we have
\begin{align*}
  x = & \sum_{i=0}^{n} \sum_{j=0}^n y_i^* y_j 
   =   \sum_{i=0}^{n} \sum_{j=0}^n \pi_K^{(i)}(h_i)^*\pi_K^{(j)}(h_j) \\
   = & \sum_{i=0}^{n} \pi_K^{(i)}
   \left( \sum_{j=0}^i (h_j \otimes id_{(i,j)})^* h_i + h_i^*
  \sum_{j=0}^{i-1}(h_j \otimes id_{(i,j)})  \right)  
   =  \sum_{i=0}^{n} \pi_K^{(i)}(k_i),
\end{align*}
where
\[
  k_i = \sum_{j=0}^i (h_j \otimes id_{(i,j)})^* h_i + h_i^*
  \sum_{j=0}^{i-1} (h_j \otimes id_{(i,j)}).
\]
We put
\[
  k =  \sum_{i=0}^{n} k_i \otimes id_{(n,i)}.
\]

\begin{align*}
k = & \sum_{i=0}^{n}\left( \sum_{j=0}^i (h_j \otimes id_{(i,j)})^* h_i +
  h_i^*
  \sum_{j'=0}^{i-1} h_{j'} \otimes id_{(i,j')}  \right) \otimes 
id_{(n,i)} \\
  = &  \sum_{i=0}^{n}\left( \sum_{j=0}^i (h_j \otimes id_{(n,j)})^*
  (h_i \otimes id_{(n,i)}) + \sum_{j'=0}^{i-1}(h_i\otimes id_{(n,i)})^*
  (h_{j'} \otimes id_{(n,j')})  \right)  \\
  = & \left( \sum_{i=0}^n (h_i \otimes id_{(n,i)})\right)^*
     \left(\sum_{j=0}^{n} (h_j \otimes id_{(n,j)})\right)\\
  \ge & 0.
\end{align*}
As in the proof of Proposition \ref{prop:well},
we put ${\bf p}=(i_1,i_2,\dots,i_n)$, ${\bf u}_{\bf p}
= u_{i_1}\otimes u_{i_2} \otimes
\cdots \otimes u_{i_n}$, ${\bf p}'=(i_1,i_2,\dots,i_n,i_{n+1})$ and
${\bf u}_{{\bf p}'}= u_{i_1}\otimes u_{i_2} \otimes\cdots \otimes u_{i_n}
\otimes u_{i_{n+1}} = {\bf u}_{\bf p} \otimes u_{i_{n+1}}$.

We prepare the following:
For $k \in \K(X^{\otimes i})$ ($0 \le i \le n$), we have
\begin{align*}
  & e^{-(n+1)\beta}  \sum_{{\bf p}' \in {\mathbb N}^{n+1}}
\tau(({\bf u}_{{\bf p}'}|
(k \otimes id_{(n+1,i)}) {\bf u}_{{\bf p}'})_A)
  \\
= & e^{-(n+1)\beta}  \sum_{{\bf p} \in {\mathbb N}^n}
     \sum_{i_{n+1}=1}^{\infty} \tau(({\bf u}_{\bf p} \otimes  u_{i_{n+1}}|
    (k \otimes id_{(n+1,i)})({\bf u}_{\bf p}
  \otimes u_{i_{n+1}}))_A) \\
  = & e^{-(n+1)\beta}  \sum_{{\bf p} \in {\mathbb N}^n} 
\sum_{i_{n+1}=1}^{\infty}
    \tau((u_{i_{n+1}} |\phi(({\bf u}_{\bf p}|
(k \otimes id_{(n,i)}){\bf u}_{\bf p})_A) u_{i_{n+1}})_A).
\end{align*}

If $\tau$ satisfies ($\beta 2$) and $T \in \cL(X^{\otimes n})$ is 
positive,
we have
\[
\sum_{i_{n+1}=1}^{\infty} \tau ((u_{i_{n+1}} | ({\bf u}_{\bf p} |
T {\bf u}_{\bf p})_Au_{i_{n+1}})_A)
  \le e^{\beta} \tau(({\bf u}_{\bf p} | T {\bf u}_{\bf p})_A).
\]
Using these, we prove the Proposition.
Let $x \in (\F^{(n)}{})^+$.  We express $x=\sum_{i=1}^nx_i$ where $x_i 
\in B_i$.
For $x_i$ we take $k_i$ such that $x_i=\pi_K^{i}(k_i)$ and write as
$k = \sum_{i=0}^n  (k_i \otimes id_{(n+1,i)})$.
Then by Lemma \ref{lem:natural} and by the fact that $\tau$ satisfies
($\beta 2$), we have
\begin{align*}
  \overline{\sigma^{(n+1)}}(x)
= & e^{-(n+1)\beta} \sum_{{{\bf p}'}\in {\mathbb N}^{n+1}}
\tau(({\bf u}_{{\bf p}'}  | \sum_{i=0}^n
  (k_i \otimes id_{(n+1,i)}) {\bf u}_{{\bf p}'})_A) \\
= & e^{-(n+1)\beta} \sum_{{{\bf p}'} \in {\bf n}^{n+1}} \tau(({\bf 
u}_{{\bf p}'}
| (k \otimes id_1) {\bf u}_{{\bf p}'})_A) \\
= & e^{-(n+1)\beta} \sum_{{\bf p} \in {\mathbb 
N}^n}\sum_{i_{n+1}=1}^{\infty}
  \tau(({\bf u}_{\bf p} \otimes u_{i_{n+1}}| ((k{\bf u}_{\bf p}) \otimes
u_{i_{n+1}}))_A)
  \\
= & e^{-(n+1)\beta} \sum_{{\bf p} \in {\mathbb N}^n} 
\sum_{i_{n+1}=1}^{\infty}
    \tau((u_{i_{n+1}} |\phi(({\bf u}_{\bf p} |
k{\bf u}_{\bf p})_A )u_{i_{n+1}} )_A ) \\
\le & e^{-n \beta} \sum_{{\bf p}\in {\mathbb N}^n}
    \tau(({\bf u}_{\bf p} | k {\bf u}_{\bf p})_A) \\
  = & \overline{\sigma^{(n)}}(x).
\end{align*}
\end{proof}

Let $A$ be a \cst -algebra, $\alpha$ be an automorphic action of one
dimensional torus $\bT$ on $A$.
$A^{\rm anal}$ denotes the set $a \in A$ such that $t \to
\alpha_t(a)$ has an analytic extension to ${\mathbb C}$.

\begin{defn}
Let $\beta > 0$.  A state $\varphi$ of $A$ is called a $\beta$-KMS state
  on $A$ with respect to $\alpha$ if
\[
  \varphi(x\alpha_{it}(y)) = \varphi(yx)
\]
for $x \in A$ and $y \in D$, where $D$ is a dense *-subalgebra
contained in $A^{\rm anal}$.
\end{defn}

We denote by $E$ the conditional expectation of  $\cO_X(J)$ onto
the fixed point algebra $\cO_X(J)^{\mathbb T}$  by the gauge action.  We denote by
$\cO_X(J)^{(n)}$ the $n$-spectral subspace with respect to the gauge
action.

\begin{lemma}(\cite{PWY} Proposition 1.3 ) \label{lem:KMS2}
Fix $\beta > 0$.  If $\varphi$ is a $\beta$-KMS state on 
$\cO_X(J)$, 
then for $x$, $y \in \cO_X(J)^{(n)}$,
\begin{equation} \label{eq:KMS2}
  \varphi(x^*y) = e^{n \beta}\varphi(yx^*). 
\end{equation}
Conversely, if a tracial state $\varphi$ on $\cO_X(J)^{\mathbb T}$
satisfies  the equation (\ref{eq:KMS2}) for 
$x$, $y \in \cO_X(J)^{(n)}$,then
$\tau \circ E$ is a $\beta$-KMS state on $\cO_X(J)$.
The correspondence is one to one and conserves extreme points.
\end{lemma}

\begin{lemma} (Exel-Laca \cite{EL} Proposition 12.5)
Let $B$ be a unital \cst -algebras, $A$  a \cst -subalgebra of
$B$ containing unit, and $I$ a closed two sided ideal of $B$
such that $B = A + I$.  Let $\varphi$ be a state on $A$ and  $\psi$
 a positive linear functional on $I$.  
We assume that $\varphi(x) = \psi (x)$
for  $x \in A \cap I$ and that 
$\overline{\psi (x)} \le \varphi (x)$  for 
$x \in A$.  Then there exists a unique state $\Phi$ on $B$ such that
$\Phi|_A = \varphi$ and $\Phi|_I = \psi$.
\label{lemma:extension}
\end{lemma}

\begin{cor}
Let $B$ be a unital \cst -algebras, $A$ be a \cst -subalgebra of
$B$ containing unit, and $I$ be a closed two sided ideal of $B$
such that $B = A + I$.  Let $\varphi$ be a tracial state on $A$ and  $\psi$
 a trace on $I$.  
We assume that $\varphi(x) = \psi (x)$
for  $x \in A \cap I$ and that 
$\overline{\psi (x)} \le \varphi (x)$  for 
$x \in A$.  Then there exists a unique tracial state $\Phi$ on $B$ such that
$\Phi|_A = \varphi$ and $\Phi|_I = \psi$.
\label{lemma:trace-extension}
\end{cor}
\begin{proof}
Let $\Phi$ be the state extension on $B$ constructed in 
Lemma \ref{lemma:extension}. 
All we have to show is that  $\Phi$ is tracial. 
Consider GNS representation 
$(\pi_{\psi},H_{\psi}, \xi_{\psi})$ of $I$. Let 
$\pi : B \rightarrow B(H_{\psi})$ be the extension of 
$\pi_{\psi}$.   
The canonical extension  
$\overline{\psi}$ of $\psi$ 
to $B$
is defined as $\overline{\psi}(b) 
= (\pi(b)\xi_{\psi} \ | \ \xi_{\psi})$ for $b \in B$. 
Define a state $\psi'$
on the von Neumann algebra $\pi(I)''$ by 
$\psi'(m) = (m\xi_{\psi} \ | \ \xi_{\psi})$, for 
$m \in \pi(I)''$. Since $\psi$ is tracial,  $\psi'$ 
is also tracial. Since 
$\overline{\psi}(b) = (\pi(b)\xi_{\psi} \ | \ \xi_{\psi})$, 
the canonical extension  $\overline{\psi}$ is also 
tracial.  Hence for $a,b \in A$ and $x,y \in I$, we have 
\begin{align*}
  \Phi((a + x)(b + y)) = & \Phi(ab + xb + ay + xy)
   = \varphi(ab) + \overline{\psi}(xb + ay + xy)  \\
   = & \varphi(ba) + \overline{\psi}(bx +ya +yx) 
   = \Phi((b + y)(a + x)), 
\end{align*}
because $\varphi$ and  $\overline{\psi}$ are tracial. 
Thus $\Phi$ is also tracial. 

\end{proof}

Under these preparations, 
we shall generalize Laca-Neshevyev's theorem of the construction of 
KMS states on Cuntz-Pimsner algebras as follows:  

\begin{thm} \label{th:KMS}
Let $X$ be a \cst -correspondence over $A$ of finite-degree type 
with degree $N=d(X)$ and  
$\{u_{i}\}_{i=1}^{\infty}$  a countable basis of $X$. 
Let $J$ be an ideal of $A$ contained in $J_X$.
Let $\beta > 0$.  Let $\varphi$ be  a $\beta$-KMS state  on
a relative Cuntz-Pimsner algebra $\O_X(J)$ 
with respect to the gauge action $\gamma$.  
Then the restriction of $\varphi$ to $\pi_A(A)$
is a tracial state on $A$ satisfying
$\beta$-condition.  Conversely,a tracial state on $A$ satisfying
$\beta$-condition extends to a $\beta$-KMS state on $\O_X(J)$.
The correspondence between the $\beta$-KMS states on $\cO_{X}(J)$ and
the tracial states on $A$ satisfying $\beta$-condition given by
$\varphi \to \varphi|_{\pi_A(A)} \circ \pi_A$ is bijective and
affine.
\end{thm}

\begin{proof}
Let $\varphi$ be a $\beta$-KMS state on $\cO_X(J)$.
The restriction of $\varphi$ to $\cO_X(J)^{\mathbb T}$ is a tracial
state and satisfies the condition (\ref{eq:KMS2}) in Lemma \ref{lem:KMS2}.
 For $a \in J$, we have $\phi(a) \in 
\K(X)$ and
$\pi_A(a) = \pi_K(\phi(a))$.
Since the equation
\[
  \sum_{i=1}^{n}\pi_X(u_i)\pi_X(u_i)^*\pi_A(a)  =
  \pi_{K}(\sum_{i=1}^{n}\theta_{u_i,u_i} \phi(a))
\]
holds and $\{\, \sum_{i=1}^{n}\theta_{u_i,u_i}\,\}_{n=1}^{\infty}$ is an
approximate unit of $\K(X)$, we have
\[
  \lim_{n \to \infty}\sum_{i=1}^{n}\pi_X(u_i)\pi_X(u_i)^*\pi_A(a)
    = \pi_A(a).
\]

Using this, we have
\begin{align*}
   \sum_{i=1}^{\infty}\varphi(\pi_A((u_i | \phi(a) u_i)_A))
   = & \sum_{i=1}^{\infty}\varphi(\pi_X(u_i)^* \pi_A(a) \pi_X(u_i)) \\
   = & e^{\beta} \sum_{i=1}^{\infty}
  \varphi(\pi_X(u_i)\pi_X(u_i)^*\pi_A(a))  \\
   = & e^{\beta} \varphi\left( \lim_{n \to \infty} \left(
 \left(\sum_{i=1}^{n}\pi_X(u_i)\pi_X(u_i)^*\right)\pi_A(a)\right)\right) \\
   = & e^{\beta} \varphi(\pi_A(a)).
\end{align*}
This shows that $\varphi \circ \pi_A$ satisfies ($\beta 1$).
Let $a \in A^+$.  We have
\begin{align*}
  \sum_{i=1}^{n} (\phi \circ \pi_A)( (u_i | \phi(a) u_i )_A )
  = & \sum_{i=1}^n \varphi(\pi_X(u_i)^*\pi_A(a)\pi_X(u_i))  \\
   = &  e^{\beta} \sum_{i=1}^n \varphi(\pi_A(a) \pi_X(u_i)\pi_X(u_i)^*)
  \\
   & =   e^{\beta}  \varphi\left(\pi_A(a)^{1/2}\left(
  \sum_{i=1}^n \pi_X(u_i)\pi_X(u_i)^*\right) \pi_A(a)^{1/2}  \right) \\
  & \le e^{\beta} \varphi(\pi_A(a)).
\end{align*}
As $n \to \infty$, we can show that $\varphi \circ \pi_A$ satisfies
($\beta 2$).

Conversely, we take a tracial state $\tau$ on $A$ satisfying ($\beta 1$)
and ($\beta 2$).  We construct a tracial state $\omega$ on
$\cO_{X}(J)^{\mathbb T}$ satisfying the condition (\ref{eq:KMS2}) in
Lemma \ref{lem:KMS2} and $\omega|_{\pi_A(A)}\circ \pi_A = \tau$ holds.

We construct a tracial state  $\omega^{(n)}$ on $\F^{(n)}$ for each
natural integer $n$ inductively.
We put $\omega^{(0)}=\tau \circ \pi_A^{-1}$.
We assume that there exists a tracial state $\omega^{(n)}$ on $\F^{(n)}$
such that
\[
  \omega^{(n)} |_{B_n} = \tau^{(n)}\circ (\pi_K^{(n)})^{-1} = \sigma^{(n)}
\]
and
\[
  \overline{\sigma^{(n)}} \le \omega^{(n)} \qquad \text{on $\F^{(n-1)}$}.
\]
By Proposition \ref{prop:well}, for $x \in \F^{(n)} \cap B_{n+1} = B_n
  \cap B_{n+1}$ we have
\begin{equation*}
  \sigma^{(n+1)}(x) = \sigma^{(n)}(x) = \omega^{(n)}(x).
\end{equation*}
For $x \in \F^{(n)}$, by Proposition \ref{prop:inequality}, we have
\begin{equation} \label{eq:n-inequality}
  \overline{\sigma^{(n+1)}}(x^*x)  \le \overline{\sigma^{(n)}}(x^*x).
\end{equation}
By the assumption of induction, for $x \in \F^{(n-1)}$,
\[
  \overline{\sigma^{(n)}}(x^*x) \le \omega^{(n)}(x^*x).
\]
Let $x \in \F^{(n)}$ be written as $x=y+z$ where $y \in \F^{(n-1)}$
and $z \in B_{n}$.
Then we have
\begin{align}\label{eq:n-inequality2}
\overline{{\sigma}^{(n)}}(x^*x) & = \overline{{\sigma}^{(n)}}(
  (y+z)^*(y+z) ) \nonumber \\
             & = \overline{{\sigma}^{(n)}}(y^*y + y^*z + z^*y + z^*z)
                   \nonumber\\
             & = \overline{{\sigma}^{(n)}}(y^*y) + \sigma^{(n)}(y^*z
            + z^*y + z^*z) \nonumber \\
             & = \overline{{\sigma}^{(n)}}(y^*y) + \omega^{(n)}(y^*z
             + z^*y + z^*z) \nonumber \\
             & \le \omega^{(n)}(y^*y) + \omega^{(n)}(y^*z + z^*y + z^*z)
                    \nonumber \\
             & = \omega^{(n)}(y^*y + y^*z + z^*y + z^*z )
                     \nonumber\\
              & = \omega^{(n)}(x^*x).
\end{align}
Using (\ref{eq:n-inequality}) and (\ref{eq:n-inequality2}),
for $x \in \F^{(n)}$ we have
\[
   \overline{\sigma^{(n+1)}}(x^*x) \le \omega^{(n)}(x^*x).
\]
By Lemma \ref{lemma:extension}, there exists a state
$\omega^{(n+1)}$ on $\F^{(n+1)} = \F^{(n)} + B_{n+1}$ such
that $\omega^{(n+1)}|_{\F^{(n)}} = \omega^{(n)}$ and
$\omega^{(n+1)}|_{B_{n+1}} = \sigma^{(n+1)}$.
Since $\omega^{(n)}$ and $\sigma^{(n+1)}$ are traces,
we have that $\omega^{(n+1)}$ is a trace by 
Lemma \ref{lemma:trace-extension}. 

We note that that $\overline{\sigma^{(n+1)}} \le \omega^{(n+1)}$ on 
$\F^{(n)}$
because $\omega^{(n+1)}=\omega^{(n)}$ on $\F^{(n)}$.

By a mathematical induction argument, there exists a desired sequence
$\{\omega^{(n)}\}_{n=1,2,\dots}$ of tracial states on $\F^{(n)}$.
We define $\omega$ on  $\bigcup_{n=0}^{\infty} \F^{(n)}$ by
$\omega|_{\F^{(n)}} =\omega^{(n)}$, and extend it to
the closure $\cO_X(J)^{\mathbb T}$ by continuity.
Then $\omega$ is a tracial state on $\cO_X(J)^{\mathbb T}$.

Since $\omega(\pi_A(a) + \pi_K^{(1)}(k_1) + \cdots + \pi_K^{(n)}(k_n))
= \tau(a) + \tau^{(1)}(k_1)+ \cdots
  + \tau^{(n)}(k_n)$ for $a \in A$ and $k_i \in \K(X^{\otimes i})$,
$\omega$ does not depend on the choice of the basis
  $\{u_k\}_{k=1}^{\infty}$ we have used in the construction.

 From $\omega(\theta_{x_1,\cdots,x_n,y_1,\cdots,y_n}) = e^{-n \beta}\tau
  ((y_1 \otimes \cdots \otimes y_n | x_1 \otimes \cdots \otimes x_n)_A)$,
$\omega$ satisfies the condition (\ref{eq:KMS2}) of
Lemma \ref{lem:KMS2}. Let $E:\cO_X(J) \rightarrow  \cO_X(J)^{\mathbb T}$ 
be the canonical conditional expectation. Put $\varphi = \omega \circ E$. 
Then $\varphi$ is a $\beta$-KMS state of $\cO_X(J)$ such that 
its restriction to $A$ is $\tau$. 
\end{proof}

\section{KMS states on the \cst -algebra associated with finite graphs}

 Cuntz-Krieger algebras are 
generalized as graph \cst -algebras associated with general graphs 
having sinks and sources, which are  studied for example in 
Kumujian-Pask-Raeburn \cite{KPR}, 
Kumujian-Pask-Raeburn-Renault \cite{KPRR} and Fowler-Laca-Raeburn 
\cite{FLR}.  
As in Katsura \cite{Kat2}
\cst -algebras associated with graphs with
possibly sources and sinks are expressed as \cst -algebras associated
with \cst -correspondences canonically constructed from graphs. But
the left actions are not necessarily injective,  

Using the construction and Theorem \ref{th:KMS} in the preceding
section, we can describe KMS states on finite-graph \cst -algebras. 

Let $E=(E^0,E^1)$ be a finite graph without multiple edges.
We denote by $s$ the source map and by $r$ the range map of $E$. 
A vertex $v \in E^0$ is called a sink if $s^{-1}(v) = \emptyset$ and 
$v \in E^0$ is called a source if $r^{-1}(v) = \emptyset$

\begin{defn}
The graph \cst -algebra C${}^*(E)$ of a finite graph $E$ is
the universal \cst -algebra generated by mutually orthogonal projections
$\{p_v\}_{v \in E^{0}}$ and partial isometries $\{q_e\}_{e \in E^1}$
with orthogonal ranges, such that $q_e^*q_e = p_{r(e)}$, $q_eq_e^*
\le p_{s(e)}$ for $e \in E^1$ and
\[
  p_v = \sum_{e \in s^{-1}(v)} q_eq_e^* \qquad \text{ for \ } 0 < 
|s^{-1}(v)|.
\]
\end{defn}

We put
\[
  E^{0}_r = \{\, v \in E^{(0)}\,\, |\,\, |s^{-1}(v)| > 0 \,\}, \quad
  E^{0}_s = \{\, v \in E^{(0)} \,\,|\,\, |s^{-1}(v)| = 0\,\}.
\]

We note that $E^{0}_s$ is the set of sinks of $E$

Let $A = {\rm C}(E^0)$ and $X={\rm C}(E^1)$.
For $\xi$, $\eta \in X$ and $f \in A$, we put
\begin{align*}
  (\xi | \eta)_A(v) & = \sum_{e \in r^{-1}(v)}
   \overline{\xi(e)} \eta(e) \qquad \forall v \in E^0 \\
  (\xi f)(e) & = \xi(e) f(r(e)) \qquad \forall e \in E^1.
\end{align*}
Then $X$ is a Hilbert \cst -module over $A$.

We define $\phi(f)$ for $f \in A$ by
\[
  \phi(f) \xi (e) = f(s(e)) \xi(e),
\]
where $\xi \in X$.  Then $\phi$ is a *-representation
of $A$ in $\cL(X_A)$ and
$(X,\phi)$ is a \cst -correspondence over $A$.
As in atsura \cite{Kat2}, it holds that $\phi^{-1}(K(X)) = {\rm C}(E^0) =A$,
$\ker (\phi) = {\rm C}(E^{0}_s)$ and $J_X={\rm C}(E^{0}_r)$.
The left action $\phi$ on $X$ is injective if $E$ has no sinks. The 
 \cst -correspondence $X$ is full if $E$ has no sources.  
We denote by $\cO_{E}$ the (relative) Cuntz-Pimsner algebra of 
the \cst -correspondence $X$ with $J = J_X$. 
Then $\cO_{E}$ is isomorphic to the graph \cst -algebra ${\rm C}^*(E)$
(\cite{Kat1}).

We denote by $\gamma$ the gauge action of ${\mathbb T}$ on $\cO_{E}$.
Let $\beta$ be a positive number.
We consider $\beta$-KMS states on the \cst -algebra
$\cO_{E}$ with respect to the gauge action $\gamma$.

We number vertices of $E$ from $1$ to $n$.
Let $e$ be an edge such that $s(e)=i$ and $r(e) = j$.  Since multiples 
edges
are not permitted, we write $e$ as $(i,j)$.
We denote by $\chi_{(i,j)}$ the characteristic function of the edge 
$(i,j)$.
Then $\{\chi_{(i,j)} |  (i,j) \in E^{1} \}$ constitutes a basis of $X$ 
on $A$.

By Theorem \ref{th:KMS}, there exists a bijective correspondence between
the $\beta$-KMS states on $\cO_{E}$ and the tracial states $\tau$ on the
commutative \cst -algebra $A$ such that
\begin{align*}
\sum_{(i,j) \in E^1}\tau((\chi_{(i,j)} | \phi(f)\chi_{(i,j)})_A) & = 
e^{\beta} \tau(f)
\qquad f \in J_X ={\rm C}(E^0_r)\\
  \sum_{(i,j) \in E^1} \tau((\chi_{(i,j)} | \phi(f)\chi_{(i,j)})_A) & 
\le e^{\beta} \tau(f)
\qquad f \in A^+={\rm C}(E^0)^+.
\end{align*}
The correspondence is affine.

We denote by $D = [a_{i,j}]_{1 \le i,j \le n}$ the adjacency matrix of
the graph $E$ i.e.
\[  a_{ij}  =
  \begin{cases}
   &  1 \qquad \text{ there exists $e \in E^1$ such that $s(e) = j$,
     $r(e) = i$ } \\
   &  0 \qquad \text{ otherwise},
  \end{cases}
\]
and we denote by $\chi_j$ the characteristic function of the
vertex $j$.

Then we have
\begin{align*}
    (\chi_{(i,j)}|\phi(\chi_l) \chi_{(i,j)})_A(k)
  = & \sum_{r(e)=k}  \overline{\chi_{(i,j)}(e)} \phi(\chi_l)
     \chi_{(i,j)}(e)  \\
  = & \delta_{l,i} \delta_{j,k} a_{k,l}.
\end{align*}
We put $f = \sum_{l=1}^{n} f_l \chi_l$.  Then we have

\[
  (\chi_{(i,j)} | \phi(f) \chi_{(i,j)})_A (k)
   = \sum_{l=1}^{n} f_l \delta_{l,i} \delta_{j.k} a_{k,l}.
\]

For $\tau \in A^*$, we write as $\tau = {}^t(\tau_1,\dots,\tau_n)$.
We rewrite $(\beta 1)$ as:
\begin{align*}
    \sum_{(i,j) \in E^1} \tau((\chi_{(i,j)} | \phi(f) \chi_{(i,j)})_A)
  = & \sum_{(i,j) \in E^1} \sum_{k=1}^n \tau_k
  \sum_{l=1}^n f_l \delta_{l,i} \delta_{j,k} a_{k,l} \\
  = & \sum_{k=1}^n \sum_{l=1}^n \tau_k f_l a_{k,l} 
  =  \sum_{l=1}^n \left(\sum_{k=1}^n a_{k,l}\tau_k\right)f_l.
\end{align*}

We put $B={}^t D$.  Let $\tau$ be a tracial state on $A$.
Then $(\beta 1)$ and $(\beta 2)$ are written as
\begin{align}
  (B \tau | f) & = e^{\beta} (\tau | f) \qquad f \in {\rm C}(E^0_r)
\label{eq:ad1} \\
  (B \tau | f) & \le  e^{\beta} (\tau | f) \qquad f \in {\rm C}(E^0)^+.
\label{eq:ad2}
\end{align}

\begin{lemma}
If $E$ is a finite graph, $(\beta 1)$ implies $(\beta 2)$.
\end{lemma}
\begin{proof}
If $E$ has no sink, the Lemma is trivial.

Let $i$ be a sink and put $f = \chi_i$.
The left hand side of (\ref{eq:ad2}) is
\[
  \sum_{l=1}^{n} a_{k,l}f_l  = a_{k,i}.
\]
If $i$ is a sink then it becomes $0$.
\end{proof}

If $E$ has no sink the conditions $(\beta 1)$ and $(\beta 2)$
are made into one condition:
\[
  (\beta) \qquad (B\tau | f) = e^{\beta} (\tau | f) \qquad f \in A.
\]
Then $\tau$ is a Perron-Frobenius eigenvector and $e^{\beta}$ is
the Perron-Frobenius eigenvalue of $B$.

We assume that $E$ has a sink.
We denote by $E^0_1$ the set of vertices such that there exists an infinite
path from them, and denote by $E^0_4$ the set of vertices such that there
exists no infinite path from them.  We note that $E^{0}_1$ is not empty
if and only if the graph $E$ has a loop because $E$ is a finite graph.

We note that $E^{0} = E^0_1 \cup E^0_4$ and $E^0_1 \cap E^0_4 = 
\emptyset$.
The set $E^0_4$ contains all sinks, and all paths which start from the 
vertices
in  $E^0_4$ must end at sinks.

We note that there exists no edge from vertices in $E^0_4$ to
vertices in $E^0_1$.  If such an edge exists, it holds that
there exists an infinite path which starts from a vertex in $E^{0}_4$.

\begin{lemma}
We can number $E^{0}$ as follows:
\begin{enumerate}
  \item The numbers of vertices in $E^0_4$ are larger than that of
        vertices in $E^0_1$.
  \item There exists no edge from $j$ to $i$ where $i<j$ and $i$ and $j$ 
are
        edges in $E_4^0$.
  \item The numbers of sinks are larger than the number of vertices which
        are not sinks.
\end{enumerate}
\end{lemma}
\begin{proof}
We note that $E^0_4$ contains sinks.
Fist, we number vertices of $E$ so that the numbers of sinks are
larger than the numbers of vertices which are not sink.
We denote by $E(1)$ the graph obtained by removing sinks and edges
whose ranges are sinks.
If $E(1)$ has no sink, the proof is completed.
If $E(1)$ has sinks, we renumber vertices of $E(1)$ so
that the numbers of sinks in $E(1)$ are larger than the numbers
of vertices which are not sinks.  We get graphs $E(0)$, $E(1)$, $E(2)$,
\dots, inductively.  We has put $E(0)=E$ for the convenience.
Since $E$ is a finite graph, there exists a non negative integer $r$
such that $E(r)$ contains a sink and $E(r+1)$ is empty or $E(r+1)$
has no sink.

The vertices in $E^{0}_1$ are not removed, because
there exists a infinite path from starting from vertices
in $E^{0}_1$.  On the other hand, vertices in $E^{0}_4$ are
removed because a path starting from vertex in $E^{0}_4$
must reaches a sink, and if such a path remains, then a sink is
also remained.
\end{proof}

We denote by $F$ the graph obtained by removing vertices in $E^{0}_4$
and edges whose ranges are in $E^{0}_4$.  We call $F$ the core of the
graph $E$.

We put $E^{0}_3 = E^{0}_s$ and put $E^{0}_2 = E^{0}_4 \backslash E^{0}_3$.
Then $E^{0}$ is expressed as the disjoint union of $E^{0}_1$,
$E^{0}_2$ and $E^{0}_3$.  Using this dividing of vertices, we write
$f = {}^t[f_1\ f_2 \ f_3]$, $f_i \in C(E^0_i)$ $(i=1,2,3)$ for
$f \in A=C(E^0)$, and $\tau = {}^t[\tau_1 \ \tau_2 \ \tau_3]$ for 
a tracial state
$\tau$ of $A$.
We rewrite $(\beta 1)$ using the above block expression.
We use the notation $(\tau,f)$ of dual paring instead of $\tau(f)$.

We write the following equation
\[
  (B \tau, f) = e^{\beta} (\tau, f) \qquad f \in C(E^{0}_r).
\]
using the above block notation as follows:
\[
\left(
\begin{bmatrix}
     B_{11} & B_{12}  &  B_{13} \\  O  &  B_{22}  &  B_{23} \\
     O  &  O  & O
    \end{bmatrix}
    \begin{bmatrix} \tau_1  \\  \tau_2  \\ \tau_3 \end{bmatrix},
    \begin{bmatrix} f_1 \\ f_2 \\ 0 \end{bmatrix}  \right)
   =  \left( \begin{bmatrix}
      e^{\beta} \tau_1 \\ e^{\beta} \tau_2 \\ e^{\beta} \tau_3
     \end{bmatrix},
     \begin{bmatrix} f_1 \\ f_2 \\ 0
\end{bmatrix} \right).
\]

Since $f_1$ and $f_2$ are arbitrary, we have
\begin{align}
   B_{11}\tau_1 +B_{12}\tau_2 + B_{13} \tau_3 & = e^{\beta} \tau_1 
\label{first}\\
   B_{22}\tau_2 + B_{23}\tau_3  & = e^{\beta} \tau_2.  \label{second}
\end{align}

 From (\ref{second}), we have
\[
  \tau_2 = e^{-\beta} (B_{22}\tau_2 + B_{23}\tau_3).
\]
Since $(i,j)$ element in $B_{22}$ is $0$ for $i>j$, we can determine
all elements of $\tau_2$ for a given nonnegative $\tau_3$ for every
$\beta > 0$.

We assume $E$ has a loop, and $E^0_1$ is not empty.  We note that
$B_{11}$ is the transpose of the adjacency matrix of $F$.
Let $\lambda_0$ be the Perron-Frobenius eigenvalue of $B_{11}$.
Using (\ref{first}), we have
\[
   (e^{\beta} I - B_{11})\tau_1  = B_{12}\tau_2 + B_{13} \tau_3.
\]
If $e^{\beta}$ is greater than $\lambda_0$,
for nonnegative $\tau_2$, $\tau_3$ we can determine nonnegative
$\tau_1$ by
\[
  \tau_1 = (e^{\beta} I -B_{11})^{-1}(B_{12}\tau_2 + B_{13} \tau_3).
\]

For a sink $v$, let $\tau_3$ be the state corresponding to the 
Dirac measure $\delta_v$ on $v$.
 we can determine $\tau_2$ and $\tau_1$,
and we can get a tracial state $\tau_v$ by the normalization.  The tracial
state $\tau_v$ on $A$ gives the $\beta$-KMS state $\varphi_v$ on
$\cO_{E}$.

We summarize the results as the following theorem.

\begin{thm}\label{th:graph}
(1)\,\,We assume that $E$ has a loop.  We denote by $\lambda_0$
the Perron-Frobenius eigenvalue of the transposed of the adjacency
matrix of the core $F$.  If $\beta > \log \lambda_0$, the set of the 
extreme
$\beta$-KMS states on $\cO_E$ with respect to the gauge action
correspond to the Dirac measures on the set of sinks in $E$.

\noindent
(2)\,\,We assume that $E$ has no loop.  Then for every $\beta > 0$,
the set of the extreme $\beta$-KMS states on $\cO_E$ with respect
to the gauge action correspond to the Dirac measures on the set of
sinks in $E$.
\end{thm}

\begin{prop}
Each extreme KMS state in Theorem \ref{th:graph} generates a type I factor.
\end{prop}
\begin{proof}
We write as $\tau = {}^t \begin{bmatrix} \tau_1 & \tau_2 & \tau_3 
\end{bmatrix}$.
Using the equation (\ref{first}) and (\ref{second}), we can write as
\[
  (I - e^{-\beta} B)
\begin{bmatrix} \tau_1 \\ \tau_2 \\ \tau_3
\end{bmatrix}
  = \begin{bmatrix}
     0 \\ 0 \\ \tau_3.
    \end{bmatrix}
\]
Then we have
\[
  \begin{bmatrix}
   \tau_1 \\ \tau_2 \\ \tau_3
  \end{bmatrix}
= (I-e^{-\beta}B)^{-1}
  \begin{bmatrix}
     0 \\ 0 \\ \tau_3.
  \end{bmatrix}
\]
If $\beta > \log \lambda_0$, then $\sum_{i=0}^{\infty}e^{i \beta}B^i$
is convergent, and we have
\[
  \tau = \sum_{i=0}^{\infty}e^{i \beta}B^i 
\begin{bmatrix}
     0 \\ 0 \\ \tau_3.
  \end{bmatrix}.
\]

Let $v$ be a sink and  $\tau_3$ be the state corresponding to the 
Dirac measure $\delta_v$ on $v$.
We can determine $\tau_2$ and $\tau_1$,
and we  get a tracial state $\tau_v$ by the normalization of $\tau$. 

The $\beta$-KMS state $\varphi_{v}$ extending
$\tau_v$ is of finite type in \cite{LN}, and generates type I factor
(\cite{IKW}).
\end{proof}

We assume that $E$ has a loop and $E^0_1$ is not empty.
If $\beta = \log \lambda_0$, there exists a $\beta$-KMS state on $\cO_{E}$
which is of infinite type in \cite{LN}.  These KMS state are essentially
the same as that given in \cite{EFW}.

\begin{prop}\label{prop:infinite}
We assume $\beta = \log \lambda_0$.  Let $\hat{\tau}_1$ be the normalized
Perron Frobenius eigenvector of $B_{11}$.
Then $(\hat{\tau}_1, 0,0)$ is a $\beta$-KMS state on $\cO_E$.
It corresponds to a $\beta$-KMS state of the graph \cst -algebra associated
graph $F$.
\end{prop}

The KMS states in Proposition \ref{prop:infinite} generates
type III von Neumann algebra under some condition (\cite{EFW}).

\begin{rem}
  KMS states on Exel-Laca algebras are classified in \cite{EL} and
graph \cst -algebras are known to be strongly Morita equivalent
to some Exel-Laca algebra by adding tails to sinks.
But KMS states of finite graphs with sinks can not be obtained directly
from that of Exel-Laca algebras because the relation of KMS states on 
strongly
Morita equivalent \cst -algebras are not known.
\end{rem}


\end{document}